\documentclass[11pt, reqno]{amsart}

\author[M.~Balcerzak]{Marek Balcerzak}
\address{Institute of Mathematics, \L{}\'{o}d\'{z} University of Technology, ul. W\'{o}lcza\'{n}ska 215, 93-005 \L{}\'{o}d\'{z}, Poland}
\email{marek.balcerzak@p.lodz.pl}

\author[P.~Leonetti]{Paolo Leonetti}
\address{Department of Statistics, Universit\`a ``L. Bocconi'', via Roentgen 1, 20136 Milan, Italy}
\email{leonetti.paolo@gmail.com}
\urladdr{\url{https://sites.google.com/site/leonettipaolo/}} 


\title{On the Relationship between Ideal Cluster Points and Ideal Limit Points}

\usepackage[T1]{fontenc}
\usepackage{amsmath}
\usepackage{amssymb}
\usepackage{amsthm}
\usepackage[left=3.5cm, right=3.5cm, paperheight=11.8in]{geometry}
\usepackage{hyperref}
\usepackage{fancyhdr}
\usepackage{enumitem}
\usepackage{comment}
\usepackage{nicefrac}
\usepackage{mathrsfs}
\usepackage{graphicx}
\usepackage[utf8]{inputenc}
\usepackage{cancel}
\usepackage{mathtools}

\AtBeginDocument{%
   \def\MR#1{}
}

\newtheorem{thm}{Theorem}[section]
\newtheorem{cor}[thm]{Corollary}
\newtheorem{lem}[thm]{Lemma}
\newtheorem{prop}[thm]{Proposition}

\theoremstyle{definition} 
\newtheorem{defi}[thm]{Definition}
\let\olddefi\defi
\renewcommand{\defi}{\olddefi\normalfont}
\newtheorem{example}[thm]{Example}
\let\oldexample\example
\renewcommand{\example}{\oldexample\normalfont}
\newtheorem{rmk}[thm]{Remark}
\let\oldrmk\rmk
\renewcommand{\rmk}{\oldrmk\normalfont}

\newcommand{\clusterfin}{\mathrm{L}_x}

\hypersetup{
    pdftitle={On the Relationship between Ideal Cluster Points and Ideal Limit Points},
    pdfauthor={Marek Balcerzak and Paolo Leonetti},
    pdfmenubar=false,
    pdffitwindow=true,
    pdfstartview=FitH,
    colorlinks=true,
    linkcolor=blue,
    citecolor=green,
    urlcolor=cyan
}

\providecommand{\MR}[1]{}
\providecommand{\bysame}{\leavevmode\hbox to3em{\hrulefill}\thinspace}
\providecommand{\MR}{\relax\ifhmode\unskip\space\fi MR }

\providecommand{\href}[2]{#2}

\begin{document}

\maketitle
\thispagestyle{empty}

\begin{abstract}
Let $X$ be a first countable space which admits a non-trivial convergent sequence and let $\mathcal{I}$ be an analytic P-ideal. First, it is shown that the sets of $\mathcal{I}$-limit points of all sequences in $X$ are closed if and only if $\mathcal{I}$ is also an $F_\sigma$-ideal.

Moreover, let $(x_n)$ be a sequence taking values in a Polish space without isolated points. It is known that the set $A$ of its statistical limit points is an $F_\sigma$-set, the set $B$ of its statistical cluster points is closed, and that the set $C$ of its ordinary limit points is closed, with $A\subseteq B\subseteq C$. It is proved the sets $A$ and $B$ own some additional relationship: indeed, the set $S$ of isolated points of $B$ is contained also in $A$.

Conversely, if $A$ is an $F_\sigma$-set, $B$ is a closed set with a subset $S$ of isolated points such that $B\setminus S\neq \emptyset$ is regular closed, and $C$ is a closed set with $S\subseteq A\subseteq B\subseteq C$, then there exists a sequence $(x_n)$ for which: $A$ is the set of its statistical limit points, $B$ is the set of its statistical cluster points, and $C$ is the set of its ordinary limit points.

Lastly, we discuss topological nature of the set of $\mathcal{I}$-limit points when $\mathcal{I}$ is neither $F_\sigma$- nor analytic P-ideal.
\end{abstract}


\section{Introduction}\label{sec:intro}

The aim of this article is to establish some relationship between the set of ideal cluster points and the set of ideal limit points of a given sequence.

To this aim, let $\mathcal{I}$ be an ideal on the positive integers $\mathbf{N}$, i.e., a collection of subsets of $\mathbf{N}$ closed under taking finite unions and subsets. It is assumed that $\mathcal{I}$ contains the collection $\mathrm{Fin}$ of finite subsets of $\mathbf{N}$ and it is different from the whole power set $\mathcal{P}(\mathbf{N})$. Note that the family $\mathcal{I}_0$ of subsets with zero asymptotic density, that is,
$$
\mathcal{I}_0:=\left\{S\subseteq \mathbf{N}: \lim_{n\to \infty} \frac{|S\cap \{1,\ldots,n\}|}{n} =0\right\}
$$
is an ideal. Let also $x=(x_n)$ be a sequence taking values in a topological space $X$, which will be always assumed hereafter to be Hausdorff. We denote by $\Lambda_x(\mathcal{I})$ the set of $\mathcal{I}$-limit points of $x$, that is, the set of all $\ell \in X$ for which  
$
\lim_{k\to \infty}x_{n_k}=\ell,
$ 
for some subsequence $(x_{n_k})$ such that $\{n_k: k \in \mathbf{N}\} \notin \mathcal{I}$. In addition, let $\Gamma_x(\mathcal{I})$ be the set of $\mathcal{I}$-cluster points of $x$, that is, the set of all $\ell \in X$ such that $\{n: x_n \in U\} \notin \mathcal{I}$ for every neighborhood $U$ of $\ell$. Note that $\clusterfin:=\Lambda_x(\mathrm{Fin})=\Gamma_x(\mathrm{Fin})$ is the set of ordinary limit points of $x$; we also shorten $\Lambda_x:=\Lambda_x(\mathcal{I}_0)$ and $\Gamma_x:=\Gamma_x(\mathcal{I}_0)$.

Statistical limit points and statistical cluster points (i.e., $\mathcal{I}_{0}$-limit points and $\mathcal{I}_0$-cluster points, resp.) of real sequences were introduced by Fridy \cite{MR1181163}, cf. also \cite{
MR1372186, MR2463821, MR1416085, 
MR1838788, MR1821765, MR1260176}.

We are going to provide in Section \ref{sec:topology}, under suitable assumptions on $X$ and $\mathcal{I}$, a characterization of the set of $\mathcal{I}$-limit points. Recall that $\Gamma_x(\mathcal{I})$ is closed and contains $\Lambda_x(\mathcal{I})$, see e.g. \cite[Section 5]{MR2148648}. 
Then, it is shown that: 
\begin{enumerate}[label=(\roman*)]
\item \label{itema1} $\Lambda_x(\mathcal{I})$ is an $F_\sigma$-set, provided that $\mathcal{I}$ is an analytic P-ideal (Theorem \ref{thm:characterizationlimit});
\item \label{itema2} $\Lambda_x(\mathcal{I})$ is closed, provided that $\mathcal{I}$ is an $F_\sigma$-ideal (Theorem \ref{thm:fsigmaclosed});
\item \label{itema3} $\Lambda_x(\mathcal{I})$ is closed for all $x$ if and only if $\Lambda_x(\mathcal{I})=\Gamma_x(\mathcal{I})$ for all $x$ if and only if $\mathcal{I}$ is an $F_\sigma$-ideal, provided that $\mathcal{I}$ is an analytic P-ideal (Theorem \ref{thm:charac});
\item \label{itema4} For every $F_\sigma$-set $A$, there exists a sequence $x$ such that $\Lambda_x(\mathcal{I})=A$, provided that $\mathcal{I}$ is an analytic P-ideal which is not $F_\sigma$ (Theorem \ref{thm:fsigmacorrected});
\item \label{itema5} Each of isolated point $\mathcal{I}$-cluster point is also an $\mathcal{I}$-limit point (Theorem \ref{thm:noisolated}).
\end{enumerate}
In addition, we provide in Section \ref{sec:converse} some joint converse results: 
\begin{enumerate}[label=(\roman*)]
\setcounter{enumi}{5}
\item \label{itema6} Given $A\subseteq B \subseteq C\subseteq\mathbf{R}$ where $A$ is an $F_\sigma$-set, $B$ is non-empty regular closed, and $C$ is closed, then there exists a real sequence $x$ such that $\Lambda_x=A$, $\Gamma_x=B$, and $\clusterfin=C$ (Theorem \ref{thm:inclusion} and Corollary \ref{cor:polish});
\item \label{itema7} Given non-empty closed sets $B \subseteq C\subseteq\mathbf{R}$, there exists a real sequence $x$ such that $\Lambda_x(\mathcal{I})=\Gamma_x(\mathcal{I})=B$ and $\clusterfin=C$, provided $\mathcal{I}$ is an $F_\sigma$-ideal different from $\mathrm{Fin}$ (Theorem \ref{thm:conversesigma}).
\end{enumerate}
Lastly, it is shown in Section \ref{sec:rmks} that:
\begin{enumerate}[label=(\roman*)]
\setcounter{enumi}{7}
\item \label{itema8} $\Lambda_x(\mathcal{I})$ is analytic, provided that $\mathcal{I}$ is a co-analytic ideal (Proposition \ref{thm:coanalytic});
\item \label{itema9} An ideal $\mathcal{I}$ is maximal if and only if each real sequence $x$ admits at most one $\mathcal{I}$-limit point (Proposition \ref{prop:maximal} and Corollary \ref{cor:maximal}).
\end{enumerate}
We conclude by showing that there exists an ideal $\mathcal{I}$ and a real sequence $x$ such that $\Lambda_x(\mathcal{I})$ is not an $F_\sigma$-set (Example \ref{eq:nofsigma}).


\section{Topological structure of $\mathcal{I}$-limit points}\label{sec:topology}

We recall that an ideal $\mathcal{I}$ is said to be a \emph{P-ideal} if it is $\sigma$-directed modulo finite, i.e., for every sequence $(A_n)$ of sets in $\mathcal{I}$ there exists $A \in \mathcal{I}$ such that $A_n\setminus A$ is finite for all $n$; equivalent definitions were given, e.g., in \cite[Proposition 1]{MR2285579}.

By identifying sets of integers with their characteristic function, we equip 
$\mathcal{P}(\mathbf{N})$ with the Cantor-space topology and therefore we can assign the topological complexity to the ideals on $\mathbf{N}$. In particular, an ideal $\mathcal{I}$ is \emph{analytic} if it is a continuous image of a $G_\delta$-subset of the Cantor space. 
Moreover, a map $\varphi: \mathcal{P}(\mathbf{N}) \to [0,\infty]$ is a \emph{lower semicontinuous submeasure} provided that: 
  (i) $\varphi(\emptyset)=0$; 
  (ii) $\varphi(A) \le \varphi(B)$ whenever $A\subseteq B$; 
  (iii) $\varphi(A\cup B) \le \varphi(A)+\varphi(B)$ for all $A,B$; and 
  (iv) $\varphi(A)=\lim_{n}\varphi(A\cap \{1,\ldots,n\})$ for all $A$.

By a classical result of Solecki, an ideal $\mathcal{I}$ is an analytic P-ideal if and only if there exists a lower semicontinuous submeasure $\varphi$ such that 
\begin{equation}\label{eq:analyticPideal}
\mathcal{I}=\mathcal{I}_\varphi:=\{A\subseteq \mathbf{N}: \|A\|_\varphi =0\}
\end{equation}
and $\varphi(\mathbf{N})<\infty$, where $\|A\|_\varphi:=\lim_n \varphi(A\setminus \{1,\ldots,n\})$ for all $A\subseteq \mathbf{N}$, see \cite[Theorem 3.1]{MR1708146}. 
Note, in particular, that for every $n\in \mathbf{N}$ it holds
\begin{equation}\label{eq:invariance}
\|A\|_\varphi=\|A\setminus \{1,\ldots,n\}\|_\varphi.
\end{equation}
Hereafter, unless otherwise stated, an analytic P-ideal will be always denoted by $\mathcal{I}_\varphi$, where $\varphi$ stands for the associated lower semicontinuous submeasure as in \eqref{eq:analyticPideal}. 

Given a sequence $x=(x_n)$ taking values in a first countable space $X$ and an analytic P-ideal $\mathcal{I}_\varphi$, define 
\begin{equation}\label{eq:uell}
\mathfrak{u}(\ell):=\lim_{k \to \infty}\|\{n: x_n \in U_k\}\|_\varphi
\end{equation}
for each $\ell \in X$, where $(U_k)$ is a decreasing local base 
of neighborhoods 
at $\ell$. It is easy to see that the limit in \eqref{eq:uell} exists and its value is independent from the choice of $(U_k)$. 
\begin{lem}\label{lem:uppersemicontinuous}
The map $\mathfrak{u}$ is upper semi-continuous. In particular, the set 
$$
\Lambda_x(\mathcal{I}_\varphi,q):=\{\ell \in X: \mathfrak{u}(\ell)\ge q\}.
$$
is closed for every $q>0$.
\end{lem}
\begin{proof}
We need to prove that $\mathscr{U}_y:=\{\ell \in X: \mathfrak{u}(\ell)<y\}$ is open for all $y \in \mathbf{R}$ (hence $\mathscr{U}_\infty$ is open too). Clearly, $\mathscr{U}_y=\emptyset$ if $y\le 0$. Hence, let us suppose hereafter $y>0$ and $\mathscr{U}_y\neq \emptyset$. Fix $\ell \in \mathscr{U}_y$ and let $(U_k)$ be a decreasing local base 
of neighborhoods 
at $\ell$. Then, there exists $k_0 \in \mathbf{N}$ such that 
$
\|\{n: x_n \in U_k\}\|_\varphi<y
$ 
for every $k\ge k_0$. Fix $\ell^\prime \in U_{k_0}$ and let $(V_k)$ be a decreasing local base of neighborhoods at $\ell^\prime$. Fix also $k_1 \in \mathbf{N}$ such that $V_{k_1} \subseteq U_{k_0}$. It follows by the monotonicity of $\varphi$ that
$$
\|\{n: x_n \in V_k\}\|_\varphi\le \|\{n: x_n \in U_{k_0}\}\|_\varphi<y
$$
for every $k \ge k_1$. In particular, $\mathfrak{u}(\ell^\prime)<y$ and, by the arbitrariness of $\ell^\prime$, $U_{k_0} \subseteq \mathscr{U}_y$.
\end{proof}

At this point, we provide a useful characterization of the set $\Lambda_x(\mathcal{I}_\varphi)$ (without using limits of subsequences) and we obtain, as a by-product, that it is an $F_\sigma$-set. 
\begin{thm}\label{thm:characterizationlimit}
Let $x$ be a sequence taking values in a first countable space $X$ and $\mathcal{I}_\varphi$ be an analytic P-ideal. Then
\begin{equation}\label{eq:claimlambda}
\Lambda_x(\mathcal{I}_\varphi)=\{\ell \in X: \mathfrak{u}(\ell)>0\}.
\end{equation}
In particular, $\Lambda_x(\mathcal{I}_\varphi)$ is an $F_\sigma$-set.
\end{thm}
\begin{proof}
Let us suppose that there exists $\ell \in \Lambda_x(\mathcal{I}_\varphi)$ and let $(U_k)$ be a decreasing local base of neighborhoods at $\ell$. Then, there exists $A\subseteq \mathbf{N}$ such that $\lim_{n\to \infty, n \in A} x_{n}=\ell$ and $\|A\|_\varphi>0$. At this point, note that, for each $k \in \mathbf{N}$, the set $\{n\in A: x_n \notin U_k\}$ is finite, hence it follows by \eqref{eq:invariance} that $\mathfrak{u}(\ell) \ge \|A\|_\varphi>0$. 

On the other hand, suppose that there exists $\ell \in X$ such that $\mathfrak{u}(\ell)>0$. Let $(U_k)$ be a decreasing local base of neighborhoods at $\ell$ and define $\mathcal{A}_k:=\{n: x_n \in U_k\}$ for each $k \in \mathbf{N}$; note that $\mathcal{A}_k$ is infinite since $\|\mathcal{A}_k\|_\varphi\downarrow \mathfrak{u}(\ell)>0$ implies $\mathcal{A}_k\notin \mathcal{I}_\varphi$ for all $k$. 
Set for convenience $\theta_0:=0$ and define recursively the increasing sequence of integers $(\theta_k)$ so that $\theta_k$ is the smallest integer greater than both $\theta_{k-1}$ and $\min \mathcal{A}_{k+1}$ such that
$$
\varphi(\mathcal{A}_k \cap (\theta_{k-1},\theta_k]) \ge \mathfrak{u}(\ell)\left(1-\nicefrac{1}{k}\right).
$$
Finally, define 
$
\mathcal{A}:=\bigcup_{k} \left(\mathcal{A}_k \cap (\theta_{k-1},\theta_k]\right).
$ 
Since $\theta_k \ge k$ for all $k$, we obtain
$$
\varphi(\mathcal{A}\setminus \{1,\ldots,n\}) \ge \varphi(\mathcal{A}_{n+1} \cap (\theta_{n},\theta_{n+1}]) > \mathfrak{u}(\ell)\left(1-\nicefrac{1}{n}\right)
$$
for all $n$, hence $\|\mathcal{A}\|_\varphi \ge \mathfrak{u}(\ell) >0$. In addition, we have by construction $\lim_{n\to \infty, n \in \mathcal{A}}x_n=\ell$. Therefore $\ell$ is an $\mathcal{I}_\varphi$-limit point of $x$. To sum up, this proves \eqref{eq:claimlambda}.

Lastly, rewriting \eqref{eq:claimlambda} as $\Lambda_x(\mathcal{I}_\varphi)=\bigcup_{n}\Lambda_x(\mathcal{I}_\varphi,\nicefrac{1}{n})$ and considering that each $\Lambda_x(\mathcal{I}_\varphi,\nicefrac{1}{n})$ is closed by Lemma \ref{lem:uppersemicontinuous}, we conclude that $\Lambda_x(\mathcal{I}_\varphi)$ is an $F_\sigma$-set.
\end{proof}

The fact that $\Lambda_x(\mathcal{I}_\varphi)$ is an $F_\sigma$-set already appeared in \cite[Theorem 2]{MR2923430}, although with a different argument. The first result of this type was given in \cite[Theorem 1.1]{MR1838788} for the case $\mathcal{I}_\varphi=\mathcal{I}_0$ and $X=\mathbf{R}$. Later, it was extended in \cite[Theorem 2.6]{MR2463821} for first countable spaces. However, in the proofs contained in \cite{MR2923430, MR2463821} it is unclear why the constructed subsequence $(x_n: n \in \mathcal{A})$ converges to $\ell$. Lastly, Theorem \ref{thm:characterizationlimit} generalizes, again with a different argument, \cite[Theorem 3.1]{Leo17b} for the case $X$ metrizable.

A stronger result holds in the case that the ideal is $F_\sigma$. We recall that, by a classical result of Mazur, an ideal $\mathcal{I}$ is $F_\sigma$ if and only if there exists a lower semicontinuous submeasure $\varphi$ such that
\begin{equation}\label{eq:fsigma}
\mathcal{I}=\{A\subseteq \mathbf{N}: \varphi(A)<\infty\},
\end{equation}
with $\varphi(\mathbf{N})=\infty$, see \cite[Lemma 1.2]{MR1124539}.

\begin{thm}\label{thm:fsigmaclosed}
Let $x=(x_n)$ be a sequence taking values in a first countable space $X$ and let $\mathcal{I}$ be an $F_\sigma$-ideal. Then $\Lambda_x(\mathcal{I})=\Gamma_x(\mathcal{I})$. In particular, $\Lambda_x(\mathcal{I})$ is closed.
\end{thm}
\begin{proof}
Since it is known that $\Lambda_x(\mathcal{I})\subseteq \Gamma_x(\mathcal{I})$, the claim is clear if $\Gamma_x(\mathcal{I})=\emptyset$.  Hence, let us suppose hereafter that $\Gamma_x(\mathcal{I})$ is non-empty. Fix $\ell \in \Gamma_x(\mathcal{I})$ and let $(U_k)$ be a decreasing local base of neighborhoods at $\ell$. Letting $\varphi$ be a lower semicontinuous submeasure associated with $\mathcal{I}$ as in \eqref{eq:fsigma} and considering that $\ell$ is an $\mathcal{I}$-cluster point, we have $\varphi(A_k)=\infty$ for all $k \in \mathbf{N}$, where $A_k:=\{n: x_n \in U_k\}$.

Then, set $a_0:=0$ and define an increasing sequence of integers $(a_k)$ which satisfies
$$
\varphi(A_k \cap (a_{k-1},a_k]) \ge k
$$
for all $k$ (note that this is possible since $\varphi(A_k\setminus S)=\infty$ whenever $S$ is finite). At this point, set 
$A:=\bigcup_{k} A_k \cap (a_{k-1},a_k]$. It follows by the monotonocity of $\varphi$ that $\varphi(A)=\infty$, hence $A\notin \mathcal{I}$. Moreover, for each $k \in \mathbf{N}$, we have that $\{n \in A: x_n \notin U_k\}$ is finite: indeed, if $n \in A_j \cap (a_{j-1},a_j]$ for some $j\ge k$, then by construction $x_n \in U_j$, which is contained in $U_k$. Therefore $\lim_{n\to \infty,\, n \in A} x_{n} = \ell$, that is, $\ell \in \Lambda_x(\mathcal{I})$.
\end{proof}

Since summable ideals are $F_\sigma$ P-ideals, see e.g. \cite[Example 1.2.3]{MR1711328}, we obtain the following corollary which was proved in \cite[Theorem 3.4]{Leo17b}:
\begin{cor}\label{corsummable}
Let $x$ be a real sequence and let $\mathcal{I}$ be a summable ideal. Then $\Lambda_x(\mathcal{I})$ is closed. 
\end{cor}

It turns out that, within the class of analytic P-ideals, the property that the set of $\mathcal{I}$-limit points is always closed characterizes the subclass of $F_\sigma$-ideals:
\begin{thm}\label{thm:charac}
Let $X$ be a first countable space which admits a non-trivial convergent sequence. Let also $\mathcal{I}_\varphi$ be an analytic P-ideal. Then the following are equivalent:
\begin{enumerate}[label=\textup{(}\roman*\textup{)}]
\item \label{item1} $\mathcal{I}_\varphi$ is also an $F_\sigma$-ideal;
\item \label{item1b} $\Lambda_x(\mathcal{I}_\varphi)=\Gamma_x(\mathcal{I}_\varphi)$ for all sequences $x$;
\item \label{item2} $\Lambda_x(\mathcal{I}_\varphi)$ is closed for all sequences $x$;
\item \label{item3} there does not exist a partition $\{A_n: n \in \mathbf{N}\}$ of $\mathbf{N}$ such that $\|A_n\|_\varphi>0$ for all $n$ and $\lim_n\|\bigcup_{k>n}A_k\|_\varphi=0$.
\end{enumerate}
\end{thm}
\begin{proof}
\ref{item1} $\implies$ \ref{item1b} follows by Theorem \ref{thm:fsigmaclosed} and \ref{item1b} $\implies$ \ref{item2} is clear.

\bigskip

\ref{item2} $\implies$ \ref{item3} By hypothesis, there exists a sequence $(\ell_n)$ converging to $\ell \in X$ such that $\ell_n \neq \ell$ for all $n$. Let us suppose that there exists a partition $\{A_n: n \in \mathbf{N}\}$ of $\mathbf{N}$ such that $\|A_n\|_\varphi>0$ for all $n$ and $\lim_k\|\bigcup_{n\ge k}A_n\|_\varphi=0$. Then, define the sequence $x=(x_n)$ by $x_n=\ell_i$ for all $n \in A_i$. Then, we have that $\{\ell_n: n \in \mathbf{N}\}\subseteq \Lambda_x(\mathcal{I}_\varphi)$. On the other hand, since $X$ is first countable Hausdorff, it follows that for all $k \in \mathbf{N}$ there exists a neighborhood $U_k$ of $\ell$ such that 
$$
\textstyle \{n: x_n \in U_k\} \subseteq \{n: x_n=\ell_i \text{ for some }i\ge k\}=\bigcup_{n\ge k}A_n.
$$
Hence, by the monotonicity of $\varphi$, we obtain $0<\|\{n:x_n \in U_k\}\|_\varphi \downarrow 0$, i.e., $\mathfrak{u}(\ell)=0$, which implies, thanks to Theorem \ref{thm:characterizationlimit}, that $\ell \notin \Lambda_x(\mathcal{I}_\varphi)$. In particular, $\mathcal{I}_\varphi$ is not closed.

\bigskip

\ref{item3} $\implies$ \ref{item1} Lastly, assume that the ideal $\mathcal{I}_\varphi$ is not an $F_\sigma$-ideal. 
According to the proof of \cite[Theorem 3.4]{MR1708146}, cf. also \cite[pp. 342--343]{MR1416872}, this is equivalent to the existence, for each given $\varepsilon>0$, of some set $M\subseteq \mathbf{N}$ such that $0<\|M\|_\varphi \le \varphi(M)<\varepsilon$. This allows to define recursively a sequence of sets $(M_n)$ such that
\begin{equation}\label{eq:setrecursion}
\|M_n\|_\varphi>\sum_{k\ge n+1} \varphi(M_k)>0.
\end{equation}
for all $n$ and, in addition, $\sum_k \varphi(M_k)<\varphi(\mathbf{N})$. Then, it is claimed that there exists a partition $\{A_n: n \in \mathbf{N}\}$ of $\mathbf{N}$ such that $\|A_n\|_\varphi>0$ for all $n$ and $\lim_n\|\bigcup_{k> n}A_k\|_\varphi=0$. To this aim, set $M_0:=\mathbf{N}$ and define 
$
A_n:=M_{n-1}\setminus \bigcup_{k\ge n} M_k
$ 
for all $n\in \mathbf{N}$. It follows by the subadditivity and monotonicity of $\varphi$ that 
$$
\textstyle \varphi(M_{n-1} \setminus \{1,\ldots,k\}) \le \varphi(A_n \setminus \{1,\ldots,k\})+\varphi\left(\bigcup_{k\ge n}M_k\right)
$$
for all $n,k \in \mathbf{N}$; hence, by the lower semicontinuity of $\varphi$ and \eqref{eq:setrecursion}, 
$$
\textstyle \|A_n\|_\varphi \ge \|M_{n-1}\|_\varphi-\varphi\left(\bigcup_{k\ge n}M_k\right)\ge \|M_{n-1}\|_\varphi-\sum_{k\ge n}\varphi(M_k)>0
$$
for all $n \in \mathbf{N}$. Finally, again by the lower semicontinuity of $\varphi$, we get
$$
\textstyle \|\bigcup_{k> n}A_k\|_\varphi =\|\bigcup_{k\ge n}M_{k}\|_\varphi \le \varphi\left(\bigcup_{k\ge n}M_{k}\right) \le \sum_{k\ge n}\varphi(M_{k}) 
$$
which goes to $0$ as $n\to \infty$. This concludes the proof.
\end{proof}

At this point, thanks to Theorem \ref{thm:characterizationlimit} and Theorem \ref{thm:charac}, observe that, if $X$ is a first countable space which admits a non-trivial convergent sequence and $\mathcal{I}_\varphi$ is an analytic P-ideal which is not $F_\sigma$, then there exists a sequence $x$ such that $\Lambda_x(\mathcal{I}_\varphi)$ is a non-closed $F_\sigma$-set. In this case, indeed, all the $F_\sigma$-sets can be obtained:
\begin{thm}\label{thm:fsigmacorrected}
Let $X$ be a first countable space where all closed sets are separable and assume that there exists a non-trivial convergent sequence. Fix also an analytic P-ideal $\mathcal{I}_\varphi$ which is not $F_\sigma$ and let $B\subseteq X$ be a non-empty $F_\sigma$-set. Then, there exists a sequence $x$ such that $\Lambda_x(\mathcal{I}_\varphi)=B$.
\end{thm}
\begin{proof}
Let $(B_k)$ be a sequence of non-empty closed sets such that $\bigcup_k B_k=B$. Let also $\{b_{k,n}: n \in \mathbf{N}\}$ be a countable dense subset of $B_k$. Thanks to Theorem \ref{thm:charac}, there exists a partition $\{A_n: n \in \mathbf{N}\}$ of $\mathbf{N}$ such that $\|A_n\|_\varphi>0$ for all $n$ and $\lim_n \|\bigcup_{k>n}A_k\|_\varphi=0$. Moreover, for each $k \in \mathbf{N}$, set $\theta_{k,0}:=0$ and it is easily seen that there exists an increasing sequence of positive integers $(\theta_{k,n})$ such that 
$$
\varphi(A_k \cap (\theta_{k,n-1},\theta_{k,n}]) \ge \frac{1}{2}\|A_k \setminus \{1,\ldots,\theta_{k,n-1}\}\|_\varphi=\frac{1}{2}\|A_k\|_\varphi
$$
for all $n$. Hence, setting $A_{k,n}:=A_k \cap \bigcup_{m \in A_n}(\theta_{k,m-1},\theta_{k,m}]$, we obtain that $\{A_{k,n}: n\in \mathbf{N}\}$ is a partition of $A_k$ such that $\frac{1}{2}\|A_k\|_\varphi \le \|A_{k,n}\|_\varphi \le \|A_k\|_\varphi$ for all $n,k$.

At this point, let $x=(x_n)$ defined by $x_n=b_{k,m}$ whenever $n \in A_{k,m}$. Fix $\ell \in B$, then there exists $k \in \mathbf{N}$ such that $\ell \in B_k$. Let $(b_{k,r_m})$ be a sequence in $B_k$ converging to $\ell$. Thus, set $\tau_0:=0$ and let $(\tau_m)$ be an increasing sequence of positive integers such that $\varphi(A_{k,r_m} \cap (\tau_{m-1},\tau_m]) \ge \frac{1}{2}\|A_{k,r_m}\|_\varphi$ for each $m$. Setting $A:=\bigcup_m A_{k,r_m} \cap (\tau_{m-1},\tau_m]$, it follows by construction that $\lim_{n\to \infty, n \in A}x_n=\ell$ and $\|A\|_\varphi \ge \frac{1}{4}\|A_k\|_\varphi>0$. This shows that $B\subseteq \Lambda_x(\mathcal{I}_\varphi)$.

To complete the proof, fix $\ell \notin B$ and let us suppose for the sake of contradiction that there exists $A\subseteq \mathbf{N}$ such that $\lim_{n\to \infty, n \in A}x_n=\ell$ and $\|A\|_\varphi>0$. For each $m \in \mathbf{N}$, let $U_m$ be an open neighborhood of $\ell$ which is disjoint from the closed set $B_1 \cup \cdots B_m$. It follows by the subadditivity and the monotonicity of $\varphi$ that there exists a finite set $Y$ such that
$$
\textstyle \|A\|_\varphi \le \|Y\|_\varphi + \|\{n \in A: x_n \notin B_1 \cup \cdots \cup B_m\}\|_\varphi \le \|\bigcup_{k>m}A_k\|_\varphi.
$$
The claim follows by the arbitrariness of $m$ and the fact that $\lim_m \|\bigcup_{k>m}A_k\|_\varphi=0$.
\end{proof}

Note that every analytic P-ideal without the Bolzano-Weierstrass property cannot be $F_\sigma$, see \cite[Theorem 4.2]{MR2320288}. Hence Theorem \ref{thm:fsigmacorrected} applies to this class of ideals.


It was shown in \cite[Theorem 2.8 and Theorem 2.10]{MR2463821} that if $X$ is a topological space where all closed sets are separable, then for each $F_\sigma$-set $A$ and closed set $B$ there exist sequences $a=(a_n)$ and $b=(b_n)$ with values in $X$ such that 
$\Lambda_a=A$ and $\Gamma_b=B$.

As an application of Theorem \ref{thm:characterizationlimit}, we prove 
that, in general, its stronger version with $a=b$ fails (e.g., there are no real sequences $x$ such that $\Lambda_x=\{0\}$ and $\Gamma_x=\{0,1\}$). 

Here, a topological space $X$ is said to be \emph{locally compact} if for every $x \in X$ there exists a neighborhood $U$ of $x$ such that its closure $\overline{U}$ is compact, cf. \cite[Section 3.3]{MR1039321}.
\begin{thm}\label{thm:noisolated}
Let $x=(x_n)$ be a sequence taking values in a locally compact first countable space and fix an analytic P-ideal $\mathcal{I}_\varphi$. Then each isolated $\mathcal{I}_\varphi$-cluster point is also an $\mathcal{I}_\varphi$-limit point.
\end{thm}
\begin{proof}
Let us suppose for the sake of contradiction that there exists an isolated $\mathcal{I}_\varphi$-cluster point, let us say $\ell$, which is not an $\mathcal{I}_\varphi$-limit point. Let $(U_k)$ be a decreasing local base of open neighborhoods at $\ell$ 
such that $\overline{U}_1$ is compact. 
Let also $m$ be a sufficiently large integer such that $U_m \cap \Gamma_x(\mathcal{I}_\varphi)=\{\ell\}$. 
Thanks to \cite[Theorem 3.3.1]{MR1039321} the underlying space is, in particular, regular, hence there exists an integer $r>m$ such that $\overline{U}_r$ is a compact contained in $U_m$. 
In addition, since $\ell$ is an $\mathcal{I}_\varphi$-cluster point and it is not an $\mathcal{I}_\varphi$-limit point, it follows by Theorem \ref{thm:characterizationlimit} that
$$
0<\|\{n: x_n \in U_k\}\|_\varphi \downarrow \mathfrak{u}(\ell)=0.
$$ 
In particular, there exists $s \in \mathbf{N}$ such that $0<\|\{n: x_n \in U_s\}\|_\varphi<\|\{n: x_n \in U_{r}\}\|_\varphi$.

Observe that $K:=\overline{U}_{r} \setminus U_s$ is a closed set contained in $\overline{U}_1$, hence it is compact. By construction we have that $K \cap \Gamma_x(\mathcal{I}_\varphi)=\emptyset$. Hence, for each $z \in K$, there exists an open neighborhood $V_z$ of $z$ such that $V_z \subseteq U_m$ and $\{n: x_n \in V_z\} \in \mathcal{I}_\varphi$, i.e., $\|\{n: x_n \in V_z\}\|_\varphi=0$. It follows that $\bigcup_{z \in K} V_z$ is an open cover of $K$ which is contained in $U_m$. Since $K$ is compact, there exists a finite set $\{z_1,\ldots,z_t\} \subseteq K$ for which
\begin{equation}\label{eq:containedfinal}
K \subseteq V_{z_1} \cup \cdots \cup V_{z_t} \subseteq U_m.
\end{equation}

At this point, by the subadditivity of $\varphi$, it easily follows that $\|A\cup B\|_\varphi \le \|A\|_\varphi+\|B\|_\varphi$ for all $A,B\subseteq \mathbf{N}$. 
Hence we have
\begin{displaymath}
\begin{split}
\|\{n: x_n \in K\}\|_\varphi &\ge \|\{n: x_n \in \overline{U}_r\}\|_\varphi- \|\{n: x_n \in U_{s}\}\|_\varphi \\
&\ge \|\{n: x_n \in U_{r}\}\|_\varphi- \|\{n: x_n \in U_{s}\}\|_\varphi >0. 
\end{split}
\end{displaymath}
On the other hand, it follows by \eqref{eq:containedfinal} that
$$
\textstyle \|\{n: x_n \in K\}\|_\varphi \le \|\{n: x_n \in \bigcup_{i=1}^t V_{z_i}\}\|_\varphi \le \sum_{i=1}^t \|\{n: x_n \in V_{z_i}\}\|_\varphi=0. 
$$
This contradiction concludes the proof.
\end{proof}

The following corollary is immediate (we omit details):
\begin{cor}
Let $x$ be a real sequence for which $\Gamma_x$ is a discrete set. Then $\Lambda_x=\Gamma_x$.
\end{cor}

\section{Joint Converse results}\label{sec:converse}

We provide now a kind of converse of Theorem \ref{thm:noisolated}, specializing to the case of the ideal $\mathcal{I}_0$: informally, if $B$ is a sufficiently smooth closed set and $A$ is an $F_\sigma$-set containing the isolated points of $B$, then there exists a sequence $x$ such that $\Lambda_x=A$ and $\Gamma_x=B$.

To this aim, we need some additional notation: let $\mathrm{d}^\star$, $\mathrm{d}_\star$, and $\mathrm{d}$ be the upper asymptotic density, lower asymptotic density, and asymptotic density on $\mathbf{N}$, resp.; in particular, $\mathcal{I}_0=\{S\subseteq \mathbf{N}: \mathrm{d}^\star(S)=0\}$.

Given a topological space $X$, the interior and the closure of a subset $S\subseteq X$ are denoted by $S^\circ$ and $\overline{S}$, respectively; $S$ is said to be \emph{regular closed} if $S=\overline{S^\circ}$. We let the Borel $\sigma$-algebra on $X$ be $\mathcal{B}(X)$. A Borel probability measure $\mu: \mathcal{B}(X) \to [0,1]$ is said to be \emph{strictly positive} whenever $\mu(U)>0$ for all non-empty open sets $U$. Moreover, $\mu$ is \emph{atomless} if, for each measurable set $A$ with $\mu(A)>0$, there exists a measurable subset $B\subseteq A$ such that $0<\mu(B)<\mu(A)$. Then, a sequence $(x_n)$ taking values in $X$ is said to be $\mu$\emph{-uniformly distributed} whenever
\begin{equation}\label{eq:muud}
\mu(F) \ge \mathrm{d}^\star(\{n: x_n \in F\})
\end{equation}
for all closed sets $F$, cf. \cite[Section 491B]{MR2462372}.

\begin{thm}\label{thm:inclusion}
Let $X$ be a separable metric space and $\mu: \mathcal{B}(X)\to [0,1]$ be an atomless strictly positive Borel probability measure. Fix also sets $A\subseteq B\subseteq C\subseteq X$ such that $A$ is an $F_\sigma$-set, and $B,C$ are closed sets such that: (i) $\mu(B)>0$, (ii) the set $S$ of isolated points of $B$ is contained in $A$, and (iii) $B\setminus S$ is regular closed. Then there exists a sequence $x$ taking values in $X$ such that
\begin{equation}\label{eq:claim}
\Lambda_x=A,\,\,\Gamma_x=B,\,\text{ and }\,\mathrm{L}_x=C.
\end{equation}
\end{thm}
\begin{proof}
Set $F:=B\setminus S$ and note that, by the separability of $X$, $S$ at most countable. In particular, $\mu(S)=0$, hence $\mu(F)=\mu(B)>0$.

Let us assume for now that $A$ is non-empty. 
Since $A$ is an $F_\sigma$-set, there exists a sequence $(A_k)$ of non-empty closed sets such that $\bigcup_k A_k=A$. Considering that $X$ is (hereditarily) second countable, then  
every closed set is separable. 
Hence, for each $k \in \mathbf{N}$, there exists a countable set $\{a_{k,n}: n \in \mathbf{N}\}\subseteq A_k$ with closure $A_k$. Considering that $F$ is a separable metric space on its own right and that the (normalized) restriction $\mu_F$ of $\mu$ on $F$, that is,
\begin{equation}\label{eq:mub}
\mu_F: \mathcal{B}(F) \to [0,1]: Y\mapsto \frac{1}{\mu(F)}\mu(Y)
\end{equation} 
is a 
Borel probability measure, it follows by \cite[Exercise 491Xw]{MR2462372} that there exists a $\mu_F$-uniformly distributed sequence $(b_n)$ which takes values in $F$ and satisfies \eqref{eq:muud}. Lastly, let $\{c_n: n \in \mathbf{N}\}$ be a countable dense subset of $C$.

At this point, let $\mathscr{C}$ be the set of non-zero integer squares and note that $\mathrm{d}(\mathscr{C})=0$. 
For each $k \in \mathbf{N}$ define 
$\mathscr{A}_k:=\{2^kn: n\in \mathbf{N}\setminus 2\mathbf{N}\}\setminus \mathscr{C}$ and $\mathscr{B}:=\mathbf{N}\setminus (2\mathbf{N}\cup \mathscr{C})$. It follows by construction that $\{\mathscr{A}_k: k \in \mathbf{N}\} \cup \{\mathscr{B}, \mathscr{C}\}$ is a partition of $\mathbf{N}$. Moreover, each $\mathscr{A}_k$ admits asymptotic density and 
\begin{equation}\label{eq:limit}
\textstyle \lim_{n\to \infty}\mathrm{d}\left(\bigcup_{k\ge n}\mathscr{A}_k\right)=0.
\end{equation}
Finally, for each positive integer $k$, let $\{\mathscr{A}_{k,m}: m \in \mathbf{N}\}$ be the partition of $\mathscr{A}_k$ defined by $\mathscr{A}_{k,1}:=\mathscr{A}_k \cap \bigcup_{n \in \mathscr{A}_1 \cup \mathscr{B} \cup \mathscr{C}}[n!,(n+1)!)$ and $\mathscr{A}_{k,m}:=\mathscr{A}_k \cap \bigcup_{n \in \mathscr{A}_m}[n!,(n+1)!)$ for all integers $m\ge 2$. Then, it is easy to check that
$$
\mathrm{d}^\star(\mathscr{A}_{k,1})=\mathrm{d}^\star(\mathscr{A}_{k,2})=\cdots=\mathrm{d}(\mathscr{A}_k)=2^{-k-1}.
$$

Hence, define the sequence $x=(x_n)$ by 
\begin{equation}\label{eq:defxn}
x_n = \begin{cases*}
                    \,a_{k,m}        & if  $n \in \mathscr{A}_{k,m}$,  \\
                    \,b_{m}        & if $n$ is the $m$-th term of $\mathscr{B}$, \\
										\,c_{m}          & if $n$ is the $m$-th term of $\mathscr{C}$.
      \end{cases*}
\end{equation}
To complete the proof, let us verify that \eqref{eq:claim} holds true:

\bigskip

\textsc{Claim (i):} $\mathrm{L}_x=C$. Note that $x_n \in C$ for all $n \in \mathbf{N}$. Since $C$ is closed by hypothesis, then $\mathrm{L}_x \subseteq C$. On the other hand, if $\ell \in C$, then there exists a sequence $(c_n)$ taking values in $C$ converging (in the ordinary sense) to $\ell$. It follows by the definition of $(x_n)$ that there exists a subsequence $(x_{n_k})$ converging to $\ell$, i.e., $C\subseteq \mathrm{L}_x$.

\bigskip

\textsc{Claim (ii):} $\Gamma_x=B$. Fix $\ell \notin B$ and let $U$ be an open neighborhood of $\ell$ disjoint from $B$ (this is possible since, in the opposite, $\ell$ would belong to $\overline{B}=B$). 
Then, $\{n: x_n \in U\} \subseteq \mathscr{C}$, which implies that $\Gamma_x \subseteq B$. 

Note that the Borel probability measure $\mu_F$ defined in \eqref{eq:mub} is clearly atomless. 
Moreover, given an open set $U\subseteq X$ with non-empty intersection with $F$, then $U \cap F^\circ \neq \emptyset$: indeed, in the opposite, we would have $F^\circ \subseteq U^c$, which is closed, hence $F=\overline{F^\circ} \subseteq U^c$, contradicting our hypothesis. This proves that every non-empty open set $V$ (relative to $F$) contains a non-empty open set of $X$. Therefore $\mu_F$ is also strictly positive. 
With these premises, fix $\ell \in F$ and let $V$ be a open neighborhood of $\ell$ (relative to $F$). Since $(b_n)$ is $\mu_F$-uniformly distributed and $\mu_F$ is strictly positive, it follows by \eqref{eq:muud} that 
\begin{displaymath}
\begin{split}
0<\mu_F(V)=1-\mu_F(V^c)&\le 1-\mathrm{d}^\star(\{n: b_n \in V^c\})\\
&=\mathrm{d}_\star(\{n: b_n \in V\}) \le \mathrm{d}^\star(\{n: b_n \in V\}).
\end{split}
\end{displaymath}
Since $\mathrm{d}(\mathscr{B})=\nicefrac{1}{2}$, we obtain by standard properties of $\mathrm{d}^\star$ that
$$
\mathrm{d}^\star(\{n: x_n \in V\}) \ge \mathrm{d}^\star(\{n \in \mathscr{B}: x_n \in V\}) =\frac{1}{2}\mathrm{d}^\star(\{n: b_n \in V\})>0.
$$
We conclude by the arbitrariness of $V$ and $\ell$ that $F\subseteq \Gamma_x$. 

Hence, we miss only to show that $S\subseteq \Gamma_x$. To this aim, fix $\ell \in S$, thus $\ell$ is also an isolated point of $A$. Hence, there exist $k,m \in \mathbf{N}$ such that $a_{k,m}=\ell$. We conclude that $\mathrm{d}^\star(\{n: x_n \in U\}) \ge \mathrm{d}^\star(\{n: x_n=\ell\})\ge \mathrm{d}(\mathscr{A}_k)>0$ for each neighborhood $U$ of $\ell$. Therefore $B=F\cup S \subseteq \Gamma_x$.

\bigskip

\textsc{Claim (iii):} $\Lambda_x=A$. Fix $\ell \in A$, hence there exists $k \in \mathbf{N}$ for which $\ell$ belongs to the (non-empty) closed set $A_k$. Since $\{a_{k,n}: n \in \mathbf{N}\}$ is dense in $A_k$, there exists a sequence $(a_{k,r_m}: m \in \mathbf{N})$ converging to $\ell$. Recall that $x_n=a_{k,r_m}$ whenever $n \in \mathscr{A}_{k,r_m}$ for each $m \in \mathbf{N}$. Set by convenience $\theta_0:=0$ and define recursively an increasing sequence of positive integers $(\theta_m)$ such that $\theta_m$ is an integer greater than $\theta_{m-1}$ for which
$$
\mathrm{d}^\star\left(\mathscr{A}_{k,r_m} \cap (\theta_{m-1},\theta_m]\right) \ge \frac{\mathrm{d}(\mathscr{A}_k)}{2}=2^{-k-2}.
$$
Then, setting $\mathcal{A}:=\bigcup_m \mathscr{A}_{k,r_m} \cap (\theta_{m-1},\theta_m]$, we obtain that the subsequence $(x_n: n \in \mathcal{A})$ converges to $\ell$ and $\mathrm{d}^\star(\mathcal{A})>0$. In particular, $A\subseteq \Lambda_x$.

On the other hand, it is known that $\Lambda_x \subseteq \Gamma_x$, see e.g. \cite{MR1181163}. If $A=B$, it follows by Claim (\textsc{ii}) that $\Lambda_x \subseteq A$ and we are done. Otherwise, fix $\ell \in B\setminus A=F\setminus A$ and let us suppose for the sake of contradiction that there exists a subsequence $(x_{n_k})$ such that $\lim_k x_{n_k}=\ell$ and $\mathrm{d}^\star(\{n_k: k \in \mathbf{N}\})>0$. Fix a real $\varepsilon>0$. Then, thanks to \eqref{eq:limit}, there exists a sufficiently large integer $n_0$ such that $\mathrm{d}\left(\bigcup_{k>n_0}\mathscr{A}_k\right) \le \varepsilon$. In addition, since $F$ is a metric space and $\mu_F$ is atomless and strictly positive (see Claim (\textsc{ii})), we have
$$
\lim_{n\to \infty}\mu_F(V_n)=\mu_F(\{\ell\})=0,
$$
where $V_n$ is the open ball (relative to $F$) with center $\ell$ and radius $\nicefrac{1}{n}$. Hence, 
there exists a sufficiently large integer $m^\prime$ such that $0<\mu_F(V_{m^\prime})\le \varepsilon$. In addition, there exists an integer $m^{\prime\prime}$ such that $V_{m^{\prime\prime}}$ is disjoint from the closed set $A_1 \cup \cdots \cup A_{n_0}$. Then, set $V:=V_m$, where $m$ is an integer greater than $\max(m^\prime,m^{\prime\prime})$ such that $\mu_F(V)<\mu_F(V_{\max(m^\prime,m^{\prime\prime})})$. In particular, by the monotonicity of $\mu_F$, we have
\begin{equation}\label{eq:varepsilon}
0<\mu_F(V) \le \mu_F(\overline{V}) \le \mu_F(V_{m^\prime}) \le \varepsilon.
\end{equation}
At this point, observe there exists a finite set $Y$ such that
\begin{displaymath}
\begin{split}
\{n_k: k \in \mathbf{N}\} &= \{n_k: x_{n_k} \in V\} \cup Y \\
&\subseteq \textstyle \left(\bigcup_{k>n_0}\mathscr{A}_k\right) \cup \{n \in \mathscr{B}: x_n \in V\} \cup \mathscr{C} \cup Y.
\end{split}
\end{displaymath}
Therefore, by the subadditivity of $\mathrm{d}^\star$, \eqref{eq:muud}, and \eqref{eq:varepsilon}, we obtain
\begin{displaymath}
\begin{split}
\mathrm{d}^\star(\{n_k: k \in \mathbf{N}\}) &\le \varepsilon+\mathrm{d}^\star(\{n \in \mathscr{B}: x_n \in V\})\le \varepsilon+\mathrm{d}^\star(\{n \in \mathscr{B}: b_n \in V\}) \\
&\le \varepsilon+\mathrm{d}^\star(\{n \in \mathscr{B}: b_n \in \overline{V}\}) \le \varepsilon+\mu_F(\overline{V}) \le 2\varepsilon.
\end{split}
\end{displaymath}
It follows by the arbitrariness of $\varepsilon$ that $\mathrm{d}(\{n_k: k \in \mathbf{N}\})=0$, i.e., $\Lambda_x \subseteq A$.

\bigskip

To complete the proof, assume now that $A=\emptyset$. In this case, note that necessarily $S=\emptyset$, and it is enough to replace \eqref{eq:defxn} with
$$
x_n = \begin{cases*}
                    \,b_{n-\lfloor\sqrt{n}\rfloor}        & if $n\notin \mathscr{C}$, \\
										\,c_{\sqrt{n}}   & if $n \in \mathscr{C}$.
      \end{cases*}
$$
Then, it can be shown with a similar argument that $\Lambda_x=\emptyset$, $\Gamma_x=B$, and $\mathrm{L}_x=C$.
\end{proof}

It is worth noting that Theorem \ref{thm:inclusion} cannot be extended to the whole class of analytic P-ideals. 
Indeed, it follows by Theorem \ref{thm:fsigmaclosed} that if $\mathcal{I}$ is an $F_\sigma$ ideal on $\mathbf{N}$ then the set of $\mathcal{I}$-limit points is closed set, cf. also Theorem \ref{thm:conversesigma} below.

In addition, the regular closedness of $B\setminus S$ is essential in the proof of Theorem \ref{thm:inclusion}. On the other hand, there exist real sequences $x$ such that $\Gamma_x$ is the Cantor set $\mathcal{C}$ (which is a perfect set but not regular closed):
\begin{example}\label{example:unif}
Given a real $r \in [0,1)$ and an integer $b\ge 2$, we write $r$ in base $b$ as $\sum_n a_n/b^{n}$, where each $a_n$ belongs to $\{0,1,\ldots,b-1\}$ and $a_{n}=\zeta$ for all sufficiently large $n$ only if $\zeta=0$. This representation is unique.

Let $x=(x_n)$ be the sequence $(0,0,1,0,\frac{1}{2},1,0,\frac{1}{3},\frac{2}{3},1,\ldots)$. This sequence is unifomly distributed in $[0,1]$, i.e., $\mathrm{d}(\{n: x_n \in [a,b]\})=b-a$ for all $0\le a <b\le 1$, and $\Gamma_x=[0,1]$, see e.g. \cite[Example 4]{MR1181163}. Let also $T:[0,1] \to \mathcal{C}$ be the injection defined by $r\mapsto T(r)$, where if $r=\sum_n a_n/2^{n} \in [0,1)$ in base $2$ then $T(r)=\sum_n 2a_n/3^{n}$ in base $3$, and $1\mapsto 1$. Observe that $\mathcal{C}\setminus T([0,1))$ is the set of points of the type $2(1/3^{n_1}+\cdots+1/3^{n_{k-1}})+1/3^{n_{k}}$, for some non-negative integers $n_1 < \cdots < n_k$; in particular, $\overline{T([0,1))}=\mathcal{C}$. 

Since the sequence $T(x):=(T(x_n))$ takes values in the closed set $\mathcal{C}$, it is clear that $\Gamma_{T(x)} \subseteq \mathcal{C}$. On the other hand, fix $\ell \in T([0,1))$ with representation $\sum_n 2a_n/3^{n}$ in base $3$, where $a_n \in \{0,1\}$ for all $n$. For each $k$, let $U_k$ be the open ball with center $\ell$ and radius $1/3^{k}$. It follows that
\begin{displaymath}
\begin{split}
\{n: T(x_n) \in U_k\} &\supseteq \left\{n: T(x_n) \in \left[\frac{2a_1}{3}+\cdots+\frac{2a_k}{3^k},\frac{2a_1}{3}+\cdots+\frac{2a_k}{3^k}+\frac{1}{3^k}\right)\right\} \\
&= \left\{n: x_n \in \left[\frac{a_1}{2}+\cdots+\frac{a_k}{2^k},\frac{a_1}{2}+\cdots+\frac{a_k}{2^k}+\frac{1}{2^k}\right)\right\}.
\end{split}
\end{displaymath}
Since $(x_n)$ is equidistributed, then 
$\mathrm{d}^\star(\{n: T(x_n) \in U_k\}) \ge 1/2^k$ for all $k$. In particular, $\Gamma_{T(x)}$ is a closed set containing $T([0,1))$, therefore $\Gamma_{T(x)}=\mathcal{C}$.
\end{example}

Finally, we provide a sufficient condition for the existence of an atomless strictly positive Borel probability measure:
\begin{cor}\label{cor:polish}
Let $X$ be a Polish space without isolated points and fix sets $A\subseteq B\subseteq C\subseteq X$ 
such that $A$ is an $F_\sigma$-set, $B \neq \emptyset$ is regular closed, and $C$ is closed. Then, there exists a sequence $x$ taking values in $X$ which satisfies \eqref{eq:claim}.
\end{cor}
\begin{proof}
First, observe that the restriction $\tilde{\lambda}$ of the Lebesgue measure $\lambda$ on the set $\mathscr{I}:=(0,1)\setminus \mathbf{Q}$ is an atomless strictly positive Borel probability measure. Thanks to \cite[Exercise 6.2.A(e)]{MR1039321}, $X$ contains a dense subspace $D$ which is homeomorphic to $\mathbf{R}\setminus \mathbf{Q}$, which is turn is homeomorphic to $\mathscr{I}$, let us say through $\eta: D \to \mathscr{I}$. This embedding can be used to transfer the measure $\tilde{\lambda}$ to the target space by setting
\begin{equation}\label{eq:muskf}
\mu: \mathcal{B}(X)\to [0,1]: Y \mapsto \tilde{\lambda}(\eta(Y\cap D)).
\end{equation}
Lastly, since $B$ is non-empty closed regular, 
then it has no isolated points and contains an open set $U$ of $X$. In particular, considering that $\eta$ is an open map, we get by \eqref{eq:muskf} that $\mu(B) \ge \mu(U)=\tilde{\lambda}(\eta(U\cap D))>0$. The claim follows by Theorem \ref{thm:inclusion}.
\end{proof} 
Note that, in general, the condition $B \neq \emptyset$ cannot be dropped: indeed, it follows by \cite[Theorem 2.14]{MR2463821} that, if $X$ is compact, then every sequence $(x_n)$ admits at least one statistical cluster point. 

We conclude with another converse result related to ideals $\mathcal{I}$ of the type $F_\sigma$ (recall that, thanks to Theorem \ref{thm:fsigmaclosed}, every $\mathcal{I}$-limit point is also an $\mathcal{I}$-cluster point):
\begin{thm}\label{thm:conversesigma}
Let $X$ be a first countable space where all closed sets are separable and let $\mathcal{I}\neq \mathrm{Fin}$ be an $F_\sigma$-ideal. Fix also closed sets $B,C \subseteq X$ such that $\emptyset \neq B\subseteq C$. Then there exists a sequence $x$ 
such that $\Lambda_x(\mathcal{I})=\Gamma_x(\mathcal{I})=B$ and $\mathrm{L}_x=C$.
\end{thm}
\begin{proof}
By hypothesis, there exists an infinite set $I \in \mathcal{I}$. Let $\varphi$ be a lower semicontinuous submeasures associated to $\mathcal{I}$ as in \eqref{eq:fsigma}. Let $\{b_n: n \in \mathbf{N}\}$ and $\{c_n: n \in \mathbf{N}\}$ be countable dense subsets of $B$ and $C$, respectively. 
In addition, set $m_0:=0$ and let $(m_k)$ be an increasing sequence of positive integers such that $\varphi((\mathbf{N}\setminus I) \cap (m_{k-1},m_k]) \ge k$ for all $k$ (note that this is possible since $\varphi(\mathbf{N}\setminus I)=\infty$ and $\varphi$ is a lower semicontinuous submeasure). 
At this point, given a partition $\{H_n: n \in \mathbf{N}\}$ of $\mathbf{N}\setminus I$, where each $H_n$ is infinite, we set
$$
\textstyle M_k:=(\mathbf{N}\setminus I) \cap \bigcup_{n \in H_k}(m_{n-1},m_n]
$$
for all $k \in \mathbf{N}$. Then, it is easily checked that $\{M_k: k \in \mathbf{N}\}$ is a partition of $\mathbf{N}\setminus I$ with $M_k \notin \mathcal{I}$ for all $k$, and that the sequence $(x_n)$ defined by
$$
x_n = \begin{cases*}
                    \,b_{k}        & if $n\in M_k$, \\
										\,c_{k}          & if $n$ is the $k$-th term of $I$.
      \end{cases*}
$$
satisfies the claimed conditions.
\end{proof}
In particular, Theorem \ref{thm:fsigmacorrected} and Theorem \ref{thm:conversesigma} fix a gap in a result of Das \cite[Theorem 3]{MR2923430} and provide its correct version.


\section{Concluding remarks}\label{sec:rmks}

In this last section, we are interested in the topological nature of the set of $\mathcal{I}$-limit points when $\mathcal{I}$ is neither $F_\sigma$- nor analytic P-ideal.

Let $\mathcal{N}$ be the set of strictly increasing sequences of positive integers. Then $\mathcal{N}$ is a Polish space, since it is closed subspace of the Polish space $\mathbf{N}^{\mathbf{N}}$ (equipped with the product topology of the discrete topology on $\mathbf{N}$). 
Let also $x=(x_n)$ be a sequence taking values in a first countable regular space $X$ and fix an arbitrary ideal $\mathcal{I}$ on $\mathbf{N}$. For each $\ell \in X$, let $(U_{\ell,m})$ be a decreasing local base of open neighborhoods at $\ell$. Then, 
$\ell$ is an $\mathcal{I}$-limit point of $x$ if and only if there exists a sequence $(n_k) \in \mathcal{N}$ such that
\begin{equation}\label{eq:limitcdef}
\{n_k:k \in \mathbf{N}\} \notin \mathcal{I}\,\,\text{ and }\,\,\{n: x_n \notin U_{\ell,m}\} \in \mathrm{Fin} \,\text{ for all }m.
\end{equation}
Set $\mathcal{I}^c:=\mathcal{P}(\mathbf{N})\setminus \mathcal{I}$ and define the continuous function 
$
\psi: \mathcal{N} \to \{0,1\}^{\mathbf{N}}: (n_k) \mapsto \chi_{\{n_k:\, k \in \mathbf{N}\}},
$ 
where $\chi_S$ is the characteristic function of a set $S\subseteq \mathbf{N}$. Moreover, 
define
$$
\zeta_m: \mathcal{N} \times X \to \{0,1\}^{\mathbf{N}}: (n_k) \times \ell \mapsto \chi_{\{n:\, x_n \notin U_{\ell,m}\}}
$$
for each $m$. Hence, it easily follows by \eqref{eq:limitcdef} that
$$
\Lambda_x(\mathcal{I})=\pi_X \left(\bigcap_m \left(\psi^{-1}(I^c) \times X \,\cap \, \zeta_m^{-1}(\mathrm{Fin})\right)\right),
$$
where $\pi_X: \mathcal{N}\times X \to X$ stands for the projection on $X$.

\begin{prop}\label{thm:coanalytic}
Let $x=(x_n)$ be a sequence taking values in a first countable regular space $X$ and let $\mathcal{I}$ be a co-analytic ideal. Then $\Lambda_x(\mathcal{I})$ is analytic.
\end{prop}
\begin{proof}
For each $(n_k) \in \mathcal{N}$ and $\ell \in X$, the sections $\zeta_m((n_k),\cdot)$ and $\zeta_m(\cdot,\ell)$ are continuous. Hence, thanks to \cite[Theorem 3.1.30]{MR1619545}, each function $\zeta_m$ is Borel measurable. Since $\mathrm{Fin}$ is an $F_\sigma$-set, we obtain that each $\zeta_m^{-1}(\mathrm{Fin})$ is Borel. Moreover, since $\mathcal{I}$ is a co-analytic ideal and $\psi$ is continuous, it follows that $\psi^{-1}(\mathcal{I}^c)\times X$ is an analytic subset of $\mathcal{N}\times X$. Therefore $\Lambda_x(\mathcal{I})$ is the projection on $X$ of the analytic set $\bigcap_m \left(\psi^{-1}(I^c) \times X \,\cap \, \zeta_m^{-1}(\mathrm{Fin})\right)$, which proves the claim. 
\end{proof}

The situation is much different for \emph{maximal ideals}, i.e., ideals which are maximal with respect to inclusion. In this regard, we recall if $\mathcal{I}$ is a maximal ideal then every bounded real sequence $x$ is $\mathcal{I}$-convergent, i.e., there exists $\ell \in \mathbf{R}$ such that $\{n: |x_n-\ell|\ge \varepsilon\} \in \mathcal{I}$ for every $\varepsilon>0$, cf. \cite[Theorem 2.2]{MR2181783}.

Let $B(a,r)$ the open ball with center $a$ and radius $r$ in a given metric space $(X,d)$, and denote by $\mathrm{diam}\, S$ the diameter of a non-empty set $S\subseteq X$, namely, $\sup_{a,b \in S}d(a,b)$. Then, the metric space is said to be \emph{smooth} if
\begin{equation}\label{eq:openclosedball}
\lim_{k\to \infty}\sup_{a \in X} \mathrm{diam} \,\overline{B(a,\nicefrac{1}{k})}=0.
\end{equation}
Note that \eqref{eq:openclosedball} holds if, e.g., the closure of each open ball $B(a,r)$ coincides with the corresponding closed ball $\{b \in X: d(a,b) \le r\}$.
\begin{prop}\label{prop:maximal}
Let $x$ be a sequence taking values in a smooth compact metric space $X$ 
and let $\mathcal{I}$ be a maximal ideal. Then $x$ has exactly one $\mathcal{I}$-cluster point. In particular, $\Lambda_x(\mathcal{I})$ is closed.
\end{prop}
\begin{proof}
Since $X$ is a compact metric space, 
then $X$ is totally bounded, i.e., for each $\varepsilon>0$ there exist finitely many open balls with radius $\varepsilon$ covering $X$. Moreover, it is well known that an ideal $\mathcal{I}$ is maximal if and only if either $A \in \mathcal{I}$ or $A^c \in \mathcal{I}$ for every $A\subseteq \mathbf{N}$.

Hence, fix $k \in \mathbf{N}$, let $\{B_{k,1},\ldots,B_{k,m_k}\}$ be a cover of $X$ of open balls with radius $\nicefrac{1}{k}$, and define $\mathscr{C}_{k,i}:=\{n: x_n \in C_{k,i}\}$ for each $i\le m_k$, where $C_{k,i}:=B_{k,i}\setminus (B_{k,1} \cup \cdots \cup B_{k,i-1})$ and $B_{k,0}:=\emptyset$. 
Considering that 
$\{\mathscr{C}_{k,1},\ldots,\mathscr{C}_{k,m_k}\}$ is a partition of $\mathbf{N}$, 
it follows by the above observations that there exists a unique $i_k \in \{1,\ldots,m_k\}$ for which $\mathscr{C}_{k,i_k} \notin \mathcal{I}$.

At this point, let $(G_k)$ be the decreasing sequence of closed sets defined by 
$$
G_k:=\overline{C_{1,i_1} \cap \cdots \cap C_{k,i_k}}
$$
for all $k$. Note that each $G_k$ is non-empty, the diameter of $G_k$ (which is contained in $\overline{B_{k,i_k}}$) 
goes to $0$ as $k\to \infty$, and 
$\{n: x_n \in G_k\} \notin \mathcal{I}$ 
for all $k$. Since $X$ is a compact metric space, 
then $\bigcap_k G_k$ is a singleton $\{\ell\}$. Considering that every open ball with center $\ell$ contains some $G_k$ with $k$ sufficiently large, it easily follows that $\Gamma_x(\mathcal{I})=\{\ell\}$. In particular, since each $\mathcal{I}$-limit point is also an $\mathcal{I}$-cluster point, we conclude that $\Lambda_x(\mathcal{I})$ is either empty or the singleton $\{\ell\}$.
\end{proof}

\begin{cor}\label{cor:maximal}
An ideal $\mathcal{I}$ is maximal if and only if every real sequence $x$ has at most one $\mathcal{I}$-limit point.
\end{cor}
\begin{proof}
First, let us assume that $\mathcal{I}$ is a maximal ideal. Let us suppose that there exists $k>0$ such that $A_k:=\{n: |x_n|>k\} \in \mathcal{I}$ and define a sequence $y=(y_n)$ by $y_n=k$ if $n \in A_k$ and $y_n=x_n$ otherwise. Then, it follows by \cite[Theorem 4]{MR2923430} and Proposition \ref{prop:maximal} that there exists $\ell \in \mathbf{R}$ such that $\Lambda_x(\mathcal{I})=\Lambda_y(\mathcal{I})\subseteq \Gamma_y(\mathcal{I})=\{\ell\}$. 

Now, assume that $A_k^c 
\in \mathcal{I}$ for all $k \in \mathbf{N}$. Hence, letting $z=(z_n)$ be the sequence defined by $z_n=x_n$ if $n \in A_k$ and $z_n=k$ otherwise, we obtain 
$$
\Lambda_x(\mathcal{I})=\Lambda_z(\mathcal{I}) \subseteq \mathrm{L}_z \subseteq \mathbf{R}\setminus (-k,k).
$$
Therefore, it follows by the arbitrariness of $k$ that $\Lambda_x(\mathcal{I})=\emptyset$.

Conversely, let us assume that $\mathcal{I}$ is not a maximal ideal. Then there exists $A \subseteq \mathbf{N}$ such that $A\notin \mathcal{I}$ and $A^c \notin \mathcal{I}$. Then, the sequence $(x_n)$ defined by $x_n=\chi_A(n)$ for each $n$ has two $\mathcal{I}$-limit points.
\end{proof}

We conclude by showing that there exists an ideal $\mathcal{I}$ and a real sequence $x$ such that $\Lambda_x(\mathcal{I})$ is not an $F_\sigma$-set.

\begin{example}\label{eq:nofsigma}
Fix a partition $\{P_m: m \in \mathbf{N}\}$ of $\mathbf{N}$ such that each $P_m$ is infinite. Then, define the ideal 
$$
\mathcal{I}:=\{A \subseteq \mathbf{N}: \{m: A \cap P_m \notin \mathrm{Fin}\} \in \mathrm{Fin}\},
$$ 
which corresponds to the Fubini product $\mathrm{Fin} \times \mathrm{Fin}$ on $\mathbf{N}^2$ (it is known that $\mathcal{I}$ is a $F_{\sigma\delta\sigma}$-ideal and it is not a P-ideal). 
Given a real sequence $x=(x_n)$, let us denote by $x \upharpoonright P_m$ the subsequence $(x_n: n \in P_m)$. Hence, a real $\ell$ is an $\mathcal{I}$-limit point of $x$ 
if and only if there exists a subsequence $(x_{n_k})$ converging to $\ell$ such that $\{n_k: k \in \mathbf{N}\} \cap P_m$ is infinite for infinitely many $m$. Moreover, for each $m$ of this type, the subsequence $(x_{n_k}) \upharpoonright P_m$ converges to $\ell$. It easily follows that
\begin{equation}\label{eq:finxfin}
\Lambda_x(\mathcal{I})=\bigcap_k \bigcup_{m\ge k} \mathrm{L}_{x \upharpoonright P_m}.
\end{equation}
(In particular, since each $\mathrm{L}_{x \upharpoonright P_m}$ is closed, then $\Lambda_x(\mathcal{I})$ is an $F_{\sigma \delta}$-set.)

At this point, let $(q_t: t \in \mathbf{N})$ be the sequence $(\nicefrac{0}{1},\nicefrac{1}{1},\nicefrac{0}{2},\nicefrac{1}{2},\nicefrac{2}{2},\nicefrac{0}{3},\nicefrac{1}{3},\nicefrac{2}{3},\nicefrac{3}{3},\ldots)$, where $q_t:=a_t/b_t$ for each $t$, 
and note that $\{q_t: t \in \mathbf{N}\}=\mathbf{Q} \cap [0,1]$. 
It follows by construction that $a_t \le b_t$ for all $t$ and $b_t = \sqrt{2t}(1+o(1))$ as $t\to \infty$. In particular, 
if $m$ is a sufficiently large integer, then
\begin{equation}\label{eq:minqi}
\min_{i \le m:\, q_i \neq q_m}\,|q_i-q_m| \ge \left(\frac{1}{\sqrt{2m}(1+o(1))}\right)^2 > \frac{1}{3m}.
\end{equation}

Lastly, for each $m \in \mathbf{N}$, define the closed set
$$
C_m:=[0,1] \cap \bigcap_{t\le m}\left(q_t-\frac{1}{2^m},q_t+\frac{1}{2^m}\right)^c.
$$
We obtain by \eqref{eq:minqi} that, if $m$ is sufficiently large, let us say $\ge k_0$, then 
$$
C_m \cup C_{m+1} = [0,1] \cap \bigcap_{t\le m}\left(q_t-\frac{1}{2^{m+1}},q_t+\frac{1}{2^{m+1}}\right)^c.
$$
It follows by induction that
$$
C_m \cup C_{m+1} \cup \cdots \cup C_{m+n} = [0,1] \cap \bigcap_{t\le m}\left(q_t-\frac{1}{2^{m+n}},q_t+\frac{1}{2^{m+n}}\right)^c.
$$
for all $n \in \mathbf{N}$. In particular, $\bigcup_{m \ge k}C_m = [0,1]\setminus \{q_1,\ldots,q_k\}$ whenever $k \ge k_0$.

Let $x$ be a real sequence such that each $\{x_n: n \in P_m\}$ is a dense subset of $C_m$. Therefore, it follows by \eqref{eq:finxfin} that
\begin{displaymath}
\Lambda_x(\mathcal{I})=\bigcap_k \bigcup_{m\ge k} C_m \subseteq  \bigcap_{k \ge k_0} \bigcup_{m\ge k} C_m = \bigcap_{k \ge k_0} \,[0,1]\setminus \{q_1,\ldots,q_k\}=[0,1]\setminus \mathbf{Q}.
\end{displaymath}
On the other hand, if a rational $q_t$ belongs to $\Lambda_x(\mathcal{I})$, then $q_t \in \bigcup_{m\ge k} C_m$ for all $k \in \mathbf{N}$, which is impossible whenever $k \ge t$. This proves that $\Lambda_x(\mathcal{I})=[0,1]\setminus \mathbf{Q}$, which is not an $F_\sigma$-set.
\end{example}


\bigskip



We leave as an open question to determine whether there exists a real sequence $x$ and an ideal $\mathcal{I}$ such that $\Lambda_x(\mathcal{I})$ is not Borel measurable.


\section*{Acknowledgments}
The authors are greatful to Szymon G\l \k{a}b (\L{}\'{o}d\'{z} University of Technology, PL) for suggesting to investigate the ideal $\mathrm{Fin} \times \mathrm{Fin}$ in Example \ref{eq:nofsigma}.


\end{document}